\documentclass[reqno, dvipdfmx]{amsart}
\usepackage{amsthm,amsmath,amssymb,amsthm}
\usepackage[left=3.5cm, right=3.5cm, top=3cm, bottom=3cm, includefoot]{geometry}
\usepackage[all]{xy}
\usepackage{color}
\usepackage{appendix}
\usepackage{tikz}
\usepackage{tikz-cd}
\usepackage{bbm}
\usepackage{mathtools}
\newtheoremstyle{mystyle}%   % Name
  {}%                         % Space above
  {}%                         % Space below
  {}%                         % Body font
  {}%                         % Indent amount
  {\bfseries}%                % Theorem head font
  {.}%                        % Punctuation after theorem head
  { }%                        % Space after theorem head, ' ', or \newline
  {}%                         % Theorem head spec (can be left empty, meaning `normal')

\theoremstyle{mystyle}
\theoremstyle{definition}
\newtheorem{theorem}{Theorem}[section]
\newtheorem{proposition}{Proposition}[section]
\newtheorem{lemma}{Lemma}[section]

\newtheorem{remark}{Remark}[section]
\numberwithin{equation}{section}

\begin{document}
\title{Conservation operator processes from asymptotic representation theory and their CLT}
\author[R. Sato]{Ryosuke SATO}
\address{Department of Mathematics, Faculty of Science, Hokkaido University, Kita 10, Nishi 8, Kita-ku, Sapporo, Hokkaido, 060-0810, Japan}
\email{r.sato@math.sci.hokudai.ac.jp}

\begin{abstract}
    In this paper, we examine applications of the theory of operator-valued processes to algebraic methods in probability theory. We show a central limit theorem for general conservation operator processes. Utilizing this, we analyze the asymptotic behavior of processes derived from unitary groups and quantum unitary groups as their ranks tend to infinity, thereby providing applications of asymptotic representation theory.
\end{abstract}

\maketitle

\allowdisplaybreaks{
%%%%%%%%%%%%%%%%%%%%%%%%%%%%%%%%%%%%%%%%%%%%%%%%%%%%%%%%%%%%%%%%%%%%%%%%%%%%%%%%%%%%%%%%%%%%%%%%%%%%%%%%%%%%%%%%%%%%%%%%%%%%%%%%%%%%%%%%%%%%%%%%%%%%%%%%%%%%%%%%%%%%
\section{Introduction}
It is well known that by second quantization, a linear operator on a Hilbert space extends to a linear operator on the symmetric Fock space, called a \emph{conservation operator}. A family of such operators indexed by a time parameter is known as a \emph{conservation operator process}\footnote{It is also often called a gauge process.}. Under suitable conditions, it has the same time-ordered moments as a compound Poisson process and thus plays a fundamental role in quantum stochastic calculus\footnote{The role of Brownian motion is played by sums of creation and annihilation operators.} (see \cite{Parthasarathy:book}). Furthermore, it naturally appears in algebraic methods of probability theory. Therefore, it is quite natural to explore applications of quantum stochastic calculus to probability theory with algebraic origin.

In this paper, we are especially interested in conservation operator processes related to the asymptotic representation theory. Let us consider the unitary groups $U(N)$ and their inductive limit $U(\infty)=\varinjlim_N U(N)$, called the \emph{infinite-dimensional unitary group}. The celebrated Edrei--Voiculescu theorem gives a complete classification of indecomposable characters of $U(\infty)$. Notably, every indecomposable character of $U(\infty)$ can be approximated by irreducible characters of $U(N)$ (see Section \ref{sec:unitary_groups} for more details). As we will discuss in this paper, from the viewpoint of non-commutative probability theory, such an approximation of characters implies a law of large numbers (LLN) for conservation operator processes derived from the center $Z(\mathfrak{gl}_N)$ of the universal enveloping algebra $U(\mathfrak{gl}_N)$. See Proposition \ref{prop:LLN}.

Our first main result, Theorem \ref{thm:general_CLT}, establishes a central limit theorem (CLT) for conservation operator processes in a general setting. Then, combining this general result with the asymptotic analysis of irreducible representations of $Z(\mathfrak{gl}_N)$, we obtain the CLT for conservation operator processes derived from $Z(\mathfrak{gl}_N)$. See Theorem \ref{thm:CLT_unitary}.

It is probably natural to compare our results with the CLTs in \cite{BB13,BB14,Kuan16}. In fact, these papers implicitly study operator-valued processes given as the representations of $U(\mathfrak{gl}_N)$ on a symmetric Fock space and show their CLTs, heavily relying on the combinatorial calculus of joint moments of such processes. In this paper, we study operator-valued processes on a symmetric Fock space, which appear through second quantization. Thus, we do not address the processes in \cite{BB13,BB14,Kuan16}, and for their CLTs, it remains unclear whether an algebraic approach based on quantum stochastic calculus can be developed. However, our result is a first attempt to apply quantum stochastic calculus in this direction.

Remarkably, our framework is directly applicable to the quantum group case. In Section \ref{sec:quantum_unitary_groups}, we discuss conservation operator processes derived from the quantized universal enveloping algebra $U_q(\mathfrak{gl}_N)$. In Proposition \ref{prop:asymptotics_q_immanants}, we study the asymptotic behavior of irreducible representations of the center $Z_q(\mathfrak{gl}_N)$ of $U_q(\mathfrak{gl}_N)$. Then, in Theorem \ref{thm:CLT_quantum_unitary}, we obtain the CLT for conservation operator processes derived from the quantum unitary groups.

% \medskip
% This paper is organized as follows: In Section \ref{sec:general}, we review basic facts of symmetric Fock space and conservation operator processes. Then, we show our first main result, Theorem \ref{thm:general_CLT}, which is a CLT for general conservation operator processes. In Section \ref{sec:unitary_groups}, we discuss the conservation operator processes derived from unitary groups and show their CLT in Theorem \ref{thm:CLT_unitary}. In addition, we extends this result to the quantum unitary groups. Especially, we study the asymptotic behavior of eigenvalues of $Z_q(\mathfrak{gl}_N)$

%%%%%%%%%%%%%%%%%%%%%%%%%%%%%%%%%%%%%%%%%%%%%%%%%%%%%%%%%%%%%%%%%%%%%%%%%%%%%%%%%%%%%%%%%%%%%%%%%%%%%%%%%%%%%%%%%%%%%%%%%%%%%%%%%%%%%%%%%%%%%%%%%%%%%%%%%%%%%%%%%%%%
\section{Conservation operator processes on symmetric Fock spaces}\label{sec:general}
In this section, we review the basic facts of symmetric Fock spaces and then study conservation operator processes, which are operator-valued processes on a symmetric Fock space. Let $\mathfrak{h}$ be a complex Hilbert space. The \emph{symmetric Fock space} $\mathcal{F}(\mathfrak{h})$ over $\mathfrak{h}$ is a Hilbert space defined as 
\[\mathcal{F}(\mathfrak{h}):=\bigoplus_{n=0}^\infty \mathfrak{h}^{\odot n},\] 
where $\mathfrak{h}^{\odot n}$ is the $n$-th symmetric tensor product space of $\mathfrak{h}$ for all $n\geq 1$, and $\mathfrak{h}^{\odot 0}:=\mathbb{C}\Omega$. Here, $\Omega$ is a unit vector, called a \emph{vacuum vector}. For any $\psi \in \mathfrak{h}$, the \emph{exponential vector} $e(\psi)\in \mathcal{F}(\mathfrak{h})$ is defined by
\[e(\psi):=\sum_{n=0}^\infty \frac{1}{\sqrt{n!}}\psi^{\odot n}.\]
We have $e(0)=\Omega$ and $\langle e(\psi), e(\varphi)\rangle=\exp(\langle \psi, \varphi\rangle)$ for any $\psi, \varphi\in \mathfrak{h}$. Moreover, the exponential vectors are linearly independent, and their linear span, denoted by $\mathcal{E}(\mathfrak{h})$, is dense in $\mathcal{F}(\mathfrak{h})$.

This fact implies the following factorizability: if $\mathfrak{h}=\mathfrak{h}_1\oplus \mathfrak{h}_2$, then the mapping 
\[e(\psi_1+\psi_2)\in \mathcal{F}(\mathfrak{h}_1\oplus\mathfrak{h}_2)\mapsto e(\psi_1)\otimes e(\psi_2)\in \mathcal{F}(\mathfrak{h}_1)\otimes \mathcal{F}(\mathfrak{h}_2) \quad (\psi_1\in \mathfrak{h}_1, \psi_2\in \mathfrak{h}_2)\] 
gives a well-defined unitary map. Namely, $\mathcal{F}(\mathfrak{h}_1\oplus \mathfrak{h}_2)\cong \mathcal{F}(\mathfrak{h}_1)\otimes \mathcal{F}(\mathfrak{h}_2)$. This factorizability plays a crucial role in quantum stochastic analysis (see \cite{Parthasarathy:book} and Remarks \ref{rem:indep_inc},\ref{rem:stat_inc} for more details).

\medskip
Let $H$ be a self-adjoint linear operator on $\mathfrak{h}$, and $\Lambda(H)$ denotes the infinitesimal generator of the one-parameter unitary group $(\lambda(e^{\mathrm{i}t H}))_{t\in \mathbb{R}}$ on $\mathcal{F}(\mathfrak{h})$ given by 
\[\lambda(e^{\mathrm{i}t H})e(\psi):= e(e^{\mathrm{i}t H}\psi) \quad (\psi\in \mathfrak{h}),\]
i.e., $\lambda(e^{\mathrm{i}tH})=e^{\mathrm{i}t \Lambda(H)}$. Let $B(\mathfrak{h})$ denote the space of all bounded linear operators on $\mathfrak{h}$. For any $H\in B(\mathfrak{h})$, we define 
\[\Lambda(H):=\Lambda\left(\frac{H+H^*}{2}\right)+\mathrm{i}\Lambda\left(\frac{H-H^*}{2\mathrm{i}}\right)\]
and call it the \emph{conservation operator}. We remark that even if $H$ is bounded, $\Lambda(H)$ is possibly unbounded, but $\mathcal{E}(\mathfrak{h})$ is contained in its domain. In particular, we have 
\begin{equation}\label{eq:inn_cons}
  \langle \Lambda(H)e(\psi), e(\varphi)\rangle = \langle H\psi, \varphi\rangle \langle e(\psi), e(\varphi)\rangle
\end{equation}
for any $\psi, \varphi\in \mathfrak{h}$. Moreover, for any $H_1, \dots, H_n\in B(\mathfrak{h})$ and $\psi\in \mathfrak{h}$, the domain of $\Lambda(H_1)$ contains $\Lambda(H_2)\cdots \Lambda(H_n)e(\psi)$. In particular, if $H_1, \dots, H_n$ are self-adjoint, for any $\psi, \varphi \in \mathfrak{h}$,
\begin{align*}
    \langle \Lambda(H_1)\cdots \Lambda(H_n)e(\psi), e(\varphi)\rangle
    & = \left.(-\mathrm{i})^n\frac{d^n}{dt_1\cdots dt_n}\right|_{t_1, \dots, t_n=0} \langle e(e^{\mathrm{i}t_1 H_1}\cdots e^{\mathrm{i}t_n H_n}\psi), e(\varphi)\rangle \nonumber\\
    & = \left.(-\mathrm{i})^n\frac{d^n}{dt_1\cdots dt_n}\right|_{t_1, \dots, t_n=0} \exp\left(\langle e^{\mathrm{i}t_1 H_1}\cdots e^{\mathrm{i}t_n H_n}\psi, \varphi \rangle \right).
\end{align*}
For instance, for any $H_1, H_2\in B(\mathfrak{h})$, we have 
\begin{equation}\label{eq:inn_cons2}
    \langle \Lambda(H_1)\Lambda(H_2)e(\psi), e(\varphi)\rangle = (\langle H_1\psi, \varphi\rangle\langle H_2\psi, \varphi\rangle +\langle H_1H_2\psi, \varphi\rangle)\langle e(\psi), e(\varphi)\rangle.
\end{equation}

\medskip
To introduce operator-valued processes, let $p$ be a non-atomic Borel spectral measure on $\mathbb{R}_{\geq 0}$, i.e., $p$ assigns the Borel $\sigma$-algebra $\mathcal{B}(\mathbb{R}_{\geq 0})$ to the orthogonal projections on $\mathfrak{h}$ and satisfies that
\begin{itemize}
    \item $p(\mathbb{R}_{\geq 0})=1$,
    \item $p(\bigcup_{n=1}^\infty E_n)=\sum_{n=1}^\infty p(E_n)$ if $\{E_n\}_{n=1}^\infty \subset \mathcal{B}(\mathbb{R}_{\geq 0})$ is mutually disjoint, where the right-hand side strongly converges,
    \item $p(\{t\})=0$ for all $t\in \mathbb{R}_{\geq 0}$.
\end{itemize}
We define $\mathfrak{h}_E:=p(E)\mathfrak{h}$ and $\psi_E:=p(E)\psi$ for any $E\in \mathcal{B}(\mathbb{R}_{\geq 0})$ and $\psi\in \mathfrak{h}$. In the literature of quantum stochastic calculus, $p$ is called a \emph{non-atomic observable} (see \cite{Parthasarathy:book}).

Let us assume that $H\in B(\mathfrak{h})$ commutes with $p$, i.e., $p([0, t])H=Hp([0, t])=:H_t$ holds for all $t\in \mathbb{R}_{\geq 0}$. The \emph{conservation operator process} $(\Lambda_t(H))_{t\geq 0}$ is defined by $\Lambda_t(H):=\Lambda(H_t)$. For any $0\leq s< t$, we define the \emph{increments} by
\[\Lambda_{[s, t)}(H):=\Lambda_t(H)-\Lambda_s(H).\] 
By Equation \eqref{eq:inn_cons}, we have 
\[\Lambda_{[s, t)}(H)e(\psi)=\Lambda(H_{[s, t)})e(\psi) \quad (\psi\in \mathfrak{h}),\]
where $H_{[s, t)}:=H_t-H_s=Hp([s, t))$. Moreover, by the factorizability of $\mathcal{F}(\mathfrak{h})$, we have

\begin{equation}\label{eq:factorizability1}
    \Lambda_{[s, t)}(H)e(\psi_{[s, t)})\in \mathcal{F}(\mathfrak{h}_{[s, t)}),
\end{equation}  
\begin{equation}\label{eq:factorizability2}
    U_{s, t}\Lambda_{[s, t)}(H)e(\psi)=e(\psi_{[0, s)})\otimes \Lambda_{[s, t)}(H)e(\psi_{[s, t)})\otimes e(\psi_{[t, \infty)}),
\end{equation}
where $U_{s, t}\colon \mathcal{F}(\mathfrak{h})\to \mathfrak{F}(\mathfrak{h}_{[0, s)})\otimes \mathfrak{F}(\mathfrak{h}_{[s, t)})\otimes \mathfrak{F}(\mathfrak{h}_{[t, \infty)})$ is a unitary operator given by 
\[U_{s, t}e(\psi):=e(\psi_{[0, s)})\otimes e(\psi_{[s, t)})\otimes e(\psi_{[t, \infty)}).\]
In fact, if $H$ is self-adjoint, for all $r\in \mathbb{R}$ we have 
\[\lambda(e^{\mathrm{i}r H_{[s, t)}})e(\psi_{[s, t)})=e(e^{\mathrm{i}r H_{[s, t)}}\psi_{[s, t)})\in \mathcal{F}(\mathfrak{h}_{[s, t)}),\]
\[U_{s, t}\lambda(e^{\mathrm{i}r H_{[s, t)}})e(\psi)=e(\psi_{[0, s)})\otimes e(e^{\mathrm{i}r H_{[s, t)}}\psi_{[s, t)})\otimes e(\psi_{[t, \infty)}).\]
Furthermore, the same statement as Equations \eqref{eq:factorizability1}, \eqref{eq:factorizability2} holds for a multitude of operators, that is, for any  $H_1, \dots, H_m\in B(\mathfrak{h})$, we have 
\begin{equation}\label{eq:factorizability3}
    \Lambda_{[s, t)}(H_1)\cdots \Lambda_{[s, t)}(H_m)e(\psi_{[s, t)})\in \mathcal{F}(\mathfrak{h}_{[s, t)}).
\end{equation}
In addition, let $0\leq s_1<t_1<s_2<\cdots <t_k$ and $I_j:=[s_j, t_j)$ for each $j=1, \dots, k$. For every $j_1, \dots, j_m\in \{1, \dots, k\}$ we have 
\begin{align}\label{eq:factorizability4}
    &U_{s_1, t_1, \dots, s_k, t_k}\Lambda_{I_{j_1}}(H_1)\cdots \Lambda_{I_{j_m}(H_m)}e(\psi) \nonumber\\
    &= e(\psi_{[0, s_1)})\otimes \left(\prod_{l: j_l=1}\Lambda_{I_1}(H_l) \right)e(\psi_{I_1})\otimes \cdots \otimes \left(\prod_{l: j_l=k}\Lambda_{I_k}(H_l)\right)e(\psi_{I_k})\otimes e(\psi_{[t_k, \infty)}),
\end{align}
where $U_{s_1, t_1, \dots, s_k, t_k}$ is a unitary map from $\mathcal{F}(\mathfrak{h})$ to $\mathcal{F}(\mathfrak{h}_{[0, s_1)})\otimes  \mathcal{F}(\mathfrak{h}_{[s_1, t_1)})\otimes \cdots \otimes\mathcal{F}(\mathfrak{h}_{[t_k, \infty)})$ defined analogously to $U_{s, t}$, and the products in the right-hand side keep the order $\Lambda_{I_{j_1}}(H_1), \dots, \Lambda_{I_{j_m}}(H_m)$.

\begin{remark}\label{rem:indep_inc}
    The factorizability in Equations \eqref{eq:factorizability3}, \eqref{eq:factorizability4} can be regarded as independence of increments. In fact, for any $\psi, \varphi\in \mathfrak{h}$, we have 
    \[\frac{\langle \Lambda_{I_{j_1}}(H_1)\cdots \Lambda_{I_{j_m}}(H_m)e(\psi), e(\varphi)\rangle }{\langle e(\psi), e(\varphi)\rangle}=\prod_{i=1}^k\frac{\langle \prod_{l: j_l=i} \Lambda_{I_i}(H_l)e(\psi_{I_i}), e(\varphi_{I_i})\rangle }{\langle e(\psi_{I_i}), e(\varphi_{I_i})\rangle},\]
    where the linear form $\langle \,\cdot\, e(\psi), e(\varphi)\rangle/\langle e(\psi), e(\varphi)\rangle$ can be regarded as expectation values of increments. Thus, the expectations of the increments of disjoint intervals are multiplicative, that is, these increments are independent.
\end{remark}

\medskip
In what follows, we focus on the following setting: let $V$ be a Hilbert space and
\[\mathfrak{h}:=L^2(\mathbb{R}_{\geq 0}; V)\cong L^2(\mathbb{R}_{\geq 0})\otimes V.\] 
For any $E\in \mathcal{B}(\mathbb{R}_{\geq 0})$, we denote by $p(E)$ the orthogonal projection onto $\mathfrak{h}_E:=L^2(E; V)\subset \mathfrak{h}$. Then, $p$ is a non-atomic observable. Moreover, every bounded linear operator $\mathsf{h}$ on $V$ yields a conservation operator process $(\Lambda_t(H))_{t\geq 0}$, where $H:=1\otimes \mathsf{h}$.

For any $\psi \in V$ and $E\in \mathcal{B}(\mathbb{R}_{\geq 0})$, we define $\psi_E:=\mathbbm{1}_E\otimes \psi\in \mathfrak{h}$ and particularly $\psi_t:=\psi_{[0, t)}$. By Equation \eqref{eq:inn_cons}, we have 
\[\frac{\langle \Lambda_{[s, t)}(H)e(\psi_T), e(\varphi_T)\rangle}{\langle e(\psi_T), e(\varphi_T)\rangle}=(t-s)\langle \mathsf{h}\psi, \varphi\rangle  =\frac{\langle \Lambda_{t-s}(H)e(\psi_{t-s}), e(\varphi_{t-s})\rangle}{\langle e(\psi_{t-s}), e(\varphi_{t-s})\rangle}\]
for any $\psi, \varphi\in V$ and $0\leq s<t\leq T$.

\begin{remark}\label{rem:stat_inc}
    Let $\mathsf{h}_i\in B(V)$ and $H_i:=1\otimes \mathsf{h}_i$ for $i=1, \dots, m$. Then, the stationary increments property of $(\Lambda_t(H_1))_{t\geq 0}, \dots, (\Lambda_t(H_m))_{t\geq 0}$ holds as follows: let $0\leq s_1<t_1<\cdots < s_k<t_k\leq T$ and $I_j:=[s_j, t_j)$. For every $j_1, \dots, j_m\in \{1, \dots, k\}$ we have 
    \[\frac{\langle \Lambda_{I_{j_1}}(H_1)\cdots \Lambda_{I_{j_m}}(H_m) e(\psi_T), e(\varphi_T)\rangle}{\langle e(\psi_T), e(\varphi_T)\rangle} = \prod_{i=1}^k \frac{\langle \prod_{l:j_l=i} \Lambda_{t_i-s_i}(H_l)e(\psi_{t_i-s_i}), e(\varphi_{t_i-s_i})\rangle}{\langle e(\psi_{t_i-s_i}), e(\varphi_{t_i-s_i})\rangle}.\]
\end{remark}

Let us fix $\psi \in \mathfrak{h}$. If $X$ is a linear operator whose domain contains $\mathcal{E}(\mathfrak{h})$, the \emph{expectation} $\langle X\rangle_\psi$ and the \emph{variance} $\mathrm{Var}_\psi(X)$ of $X$ with respect to the coherent state $e(\psi)$ are defined by 
\[ \langle X\rangle_\psi:= \frac{\langle Xe(\psi), e(\psi)\rangle}{\langle e(\psi), e(\psi)\rangle}, \quad  \mathrm{Var}_\psi(X):=\langle (X-\langle X\rangle_\psi)^*(X-\langle X\rangle_\psi)\rangle_\psi.\]

\medskip
The following are basic properties of expectation and variance.
\begin{lemma}\label{lem:variance}
    Let $\mathsf{h}\in B(V)$ and $X_t:=\Lambda_t(1\otimes \mathsf{h})$ for any $t\geq 0$. For any $t>0$ we have 
    \[\langle X_t\rangle_{\psi_t} = t \langle X_1\rangle_{\psi_1},\] 
    \[\quad \mathrm{Var}_{\psi_t}(X_t)=t\mathrm{Var}_{\psi_1}(X_1).\]
\end{lemma}
\begin{proof}
    By Equation \eqref{eq:inn_cons}, we have 
    \[\langle X_t\rangle_{\psi_t} = \langle (1\otimes\mathsf{h})_t\psi_t, \psi_t\rangle=t\langle(1\otimes \mathsf{h})\psi, \psi\rangle=t\langle X_1\rangle_{\psi_1}.\]
    Next, by Equation \eqref{eq:inn_cons2}, we have 
    \begin{align*}
        \mathrm{Var}_{\psi_t}(X_t)
        &=\langle X_t^*X_t\rangle_{\psi_t}-\overline{\langle X_t\rangle_{\psi_t}}\langle X_t\rangle_{\psi_t}\\
        &=\langle (1\otimes \mathsf{h})^*_t(1\otimes \mathsf{h})_t\psi_t, \psi_t\rangle\\
        &=t\langle \mathsf{h}^*\mathsf{h}\psi, \psi\rangle\\
        &=t\mathrm{Var}_{\psi_1}(X_1).
    \end{align*}
\end{proof}

\medskip 
Let us recall that for arbitrary family $(X^{(i)})_i$ of Gaussian random variables with zero mean, their moments can be calculated by Wick's formula: 
\[\mathbb{E}[X^{(i_1)}\cdots X^{(i_m)}]=\begin{dcases}
    \sum_{\pi\in \mathcal{P}_2(m)}\prod_{\{a, b\}\in \pi} \mathbb{E}[X^{(i_a)}X^{(i_b)}] & m \text{ is even}, \\ 0 & m \text{ is odd},
\end{dcases}\]
where $\mathcal{P}_2(m)$ is the set of pair partitions of $\{1, \dots, m\}$.

In this sense, the following theorem states that a family of conservation operator processes at time $Nt$ converges in the moment sense to a Gaussian family as $N\to \infty$. The setting is as follows: let $V_N$ be a Hilbert space and $\mathfrak{h}_N:=L^2(\mathbb{R}_{\geq 0}; V_N)$ for every $N\geq 1$. We fix $\psi_N\in V_N$ and define $\langle\, \cdot\, \rangle_{N, t} :=\langle\, \cdot\,\rangle_{(\psi_{N})_t}$ and $\mathrm{Var}_{N, t}(\,\cdot\,)$ for any $t>0$.

\begin{theorem}\label{thm:general_CLT}
    Let $(X^{(1)}_{N, t})_{t\geq 0}, \dots, (X^{(m)}_{N, t})_{t\geq 0}$ be $m$ conservation operator processes on $\mathcal{F}(\mathfrak{h}_N)$, where $X^{(i)}_{N, t}:=\Lambda_t(1\otimes \mathsf{h}^{(i)}_N)$ and $\mathsf{h}^{(i)}_N\in B(V_N)$ for $i=1, \dots, m$. Assume that $\mathsf{h}^{(i)}_N\psi_N\neq 0$ for $i=1, \dots, m$ and there exists $C>0$ such that
    \[\langle \widehat X^{(i_1)}_{N, t}\cdots \widehat X^{(i_k)}_{N, t} \rangle_{N, t}<C\]
    for every $i_1, \dots, i_k \in \{1, \dots, m\}$, $k=1, \dots, m$, and $N\geq 1$, where 
    \[\widehat X^{(i)}_{N, t}:=\frac{X^{(i)}_{N, t}-\langle X^{(i)}_{N, t} \rangle_{N, t}}{\mathrm{Var}_{N, t}(X^{(i)}_{N, t})^{1/2}} \quad (i=1, \dots, m).\] 
    Then, the family of processes $(\widehat X^{(1)}_{N, Nt})_{t\geq 0}, \dots , (\widehat X^{(m)}_{N, Nt})_{t\geq 0}$ with scaled time $Nt$ converges in the moment sense to a Gaussian family as $N\to \infty$, i.e., 
    \[\lim_{N\to \infty} \langle \widehat X^{(1)}_{N, Nt} \cdots  \widehat X^{(m)}_{N, Nt}\rangle_{N, Nt}=\begin{dcases}
        \sum_{\pi\in \mathcal{P}_2(m)} \prod_{\{a < b\}\in \pi} C_{a, b}(t) & m\text{ is even}, \\ 0 & m\text{ is odd},
    \end{dcases}\]
    where $C_{a, b}(t):=\lim_{N\to\infty}\langle \widehat X^{(a)}_{N, t} \widehat X^{(b)}_{N, t} \rangle_{N, t}$.
\end{theorem}

\begin{remark}
    By Equation \eqref{eq:inn_cons2}, we have 
    \[\langle \widehat X^{(a)}_{N, t} \widehat X^{(b)}_{N, t}\rangle_{N, t}=\frac{\langle (1\otimes \mathsf{h}^{(a)}_N)_t (1\otimes \mathsf{h}^{(b)}_N)_{t} (\psi_N)_t, (\psi_N)_t\rangle}{\|(1\otimes \mathsf{h}^{(a)}_N)(\psi_N)_t\| \|(1\otimes \mathsf{h}^{(b)}_N)(\psi_N)_t\|}=\frac{\langle \mathsf{h}^{(a)}_N \mathsf{h}^{(b)}_N \psi_N, \psi_N\rangle}{\|\mathsf{h}^{(a)}_N \psi_N\| \|\mathsf{h}^{(b)}_N \psi_N\|}.\]
    Thus, $C_{a, b}(t)$ does not depend on $t$. On the other hand, unlike the classical case, $C_{a, b}(t)\neq C_{b, a}(t)$ since $X^{(a)}_{N, t}$ and $X^{(b)}_{N, t}$ are non-commutative in general. 
\end{remark}

\begin{proof}
    By Lemma \ref{lem:variance}, we have 
    \[\langle X^{(i)}_{N, Nt}\rangle_{N, Nt}=N\langle X^{(i)}_{N, t}\rangle_{N, t}, \quad \mathrm{Var}_{N, Nt}(X^{(i)}_{N, Nt})=N\mathrm{Var}_{(\psi_N)_t}(X^{(i)}_{N, t})\]
    for each $i=1, \dots, m$. It implies that
    \begin{align*}
        \widehat X^{(i)}_{N, Nt}
        &=\frac{1}{N^{1/2}\mathrm{Var}_{N, t}(X^{(i)}_{N, t})^{1/2}}(X^{(i)}_{N, Nt}-N\langle X^{(i)}_{N, t}\rangle_{N, t})\\
        &=\frac{1}{N^{1/2}\mathrm{Var}_{N, t}(X^{(i)}_{N, t})^{1/2}} \sum_{n=0}^{N-1}(X^{(i)}_{N, [nt, (n+1)t)}-\langle X^{(i)}_{N, t}\rangle_{N, t})\\
        &=\frac{1}{N^{1/2}}\sum_{n=0}^{N-1}\widehat X^{(i)}_{N, [nt, (n+1)t)},
    \end{align*}
    where $\widehat{X}^{(i)}_{N, [nt, (n+1)t)}$ is defined similarly to $\widehat{X}^{(i)}_{N, t}$. Here, we used the following equalities
    \[\langle X^{(i)}_{N, [nt, (n+1)t)}\rangle_{N, t}=\langle X^{(i)}_{N, t}\rangle_{N, t}, \quad \mathrm{Var}_{N, t}(X^{(i)}_{N, [nt, (n+1)t)})=\mathrm{Var}_{N, t}(X^{(i)}_{N, t}).\]
    See Remark \ref{rem:stat_inc}. By the above expansion of $\widehat{X}^{(i)}_{N, Nt}$ and Remarks \ref{rem:indep_inc}, \ref{rem:stat_inc}, we have 
    \begin{align*}
        \langle \widehat X^{(1)}_{N, Nt} \cdots  \widehat X^{(m)}_{N, Nt}\rangle_{N, Nt}
        &=\frac{1}{N^{m/2}}\sum_{n_1, \dots, n_m=0}^{N-1}\langle \widehat X^{(1)}_{N, [n_1t, (n_1+1)t)} \cdots \widehat X^{(m)}_{N, [n_mt, (n_m+1)t)}\rangle_{N, Nt}\\
        &=\frac{1}{N^{m/2}}\sum_{\pi}N(N-1)\cdots (N-|\pi|+1)\prod_{\{i_1<\cdots<i_k\}\in \pi}\langle \widehat X^{(i_1)}_{N, t}\cdots \widehat X^{(i_k)}_{N, t} \rangle_{N, t},
    \end{align*}
    where the summation is over all partitions of $\{1, \dots, m\}$ into at most $N$ blocks. Since
    \[\frac{1}{N^{m/2}}N(N-1)\cdots (N-|\pi|+1)=\frac{1}{N^{m/2-|\pi|}}\left(1-\frac{1}{N}\right)\cdots \left(1-\frac{|\pi|-1}{N}\right),\]
    if $m/2>|\pi|$, the associated term converges to zero as $N\to \infty$. On the other hand, if $m/2<|\pi|$, there should exist a singleton in $\pi$, and hence, the associated term vanishes. Thus, the remaining term has to satisfy $|\pi|=m/2$, i.e., $m$ should be even, and $\pi$ is a pair partition. Therefore, we obtain the desired formula.
\end{proof}

%%%%%%%%%%%%%%%%%%%%%%%%%%%%%%%%%%%%%%%%%%%%%%%%%%%%%%%%%%%%%%%%%%%%%%%%%%%%%%%%%%%%%%%%%%%%%%%%%%%%%%%%%%%%%%%%%%%%%%%%%%%%%%%%%%%%%%%%%%%%%%%%%%
\section{Conservation operator processes derived from unitary groups}\label{sec:unitary_groups}
\subsection{The infinite-dimensional unitary group and its characters}

Let $U(N)$ be the unitary group of rank $N$. The \emph{infinite-dimensional unitary group} $U(\infty)$ is defined by $\varinjlim_N U(N)$, where $U(N)$ is naturally embedded into the upper-left corner of $U(N+1)$. We remark that the inductive limit topology of $U(\infty)$ is not locally compact. Thus, it does not possess a left-invariant Haar measure, and it causes difficulties in the Fourier analysis and the representation theory of $U(\infty)$. 

Nevertheless, the complete classification of extreme characters of $U(\infty)$ is known as the Edrei--Voiculescu theorem\footnote{In the representation theory, extreme characters of $U(\infty)$ correspond to finite factor representations of $U(\infty)$, or irreducible spherical representations of $(U(\infty)\times U(\infty), U(\infty))$. See \cite{Voiculescu76, O03} for more details.}. A complex continuous function $f$ on $U(\infty)$ is called a \emph{character} if 
\begin{itemize}
    \item (positive-definiteness) $[f(u_i^{-1}u_j)]_{i, j=1}^n$ is a positive-definite matrix for all $u_1, \dots, u_n$ in $U(\infty)$ and $n\geq 1$, 
    \item (centrality) $f(uv)=f(vu)$ for all $u, v\in G$,
    \item (normalization) $f(1)=1$.
\end{itemize}
By definition, the set $\mathrm{Ch}(U(\infty))$ of all characters of $U(\infty)$ is a convex set, and $\mathcal{E}(U(\infty))$ denotes the set of all extreme points in $\mathrm{Ch}(U(\infty))$. Then, every $f\in \mathcal{E}(U(\infty))$ has the form
\[f(u)=\prod_{z}\Phi_\omega(z),\]
where $z$ ranges over all eigenvalues of $u$ (i.e., all but finitely many of them are equal to 1), and
\[\Phi_\omega(z):=e^{\gamma^+(z-1)+\gamma^-(z^{-1}-1)}\prod_{j=1}^\infty \frac{1+\beta^+_j(z-1)}{1-\alpha^+_j(z-1)}\frac{1+\beta^-_j(z^{-1}-1)}{1-\alpha^-_j(z^{-1}-1)}\]
for the parameter $\omega=(\alpha^+, \alpha^-, \beta^+, \beta^-, \gamma^+, \gamma^-)\in (\mathbb{R}_{\geq 0}^\infty)^4\times \mathbb{R}_{\geq 0}^2$. Here, $\omega$ satisfies 
\[\alpha^\pm=(\alpha^\pm_1\geq \alpha^\pm_2\geq \dots), \quad \beta^\pm=(\beta^\pm_1\geq \beta^\pm_2\geq \cdots),\]
\[\sum_{i=1}^\infty(\alpha^\pm_i+\beta^\pm_i)<\infty, \quad \beta^+_1+\beta^-_1\leq 1.\]
Let $\Omega\subset (\mathbb{R}_{\geq 0}^\infty)^4\times (\mathbb{R}_{\geq 0})^2$ denote the set of all parameters $\omega$ satisfying the above conditions. Conversely, arbitrary $\omega\in \Omega$ gives an extreme character of $U(\infty)$, denoted by $f_\omega$, by the above formula. See \cite{BO12,Boyer92,VK82,Voiculescu76}.

The essential point in the above classification is that any extreme character of $U(\infty)$ can be \emph{approximated by} irreducible characters of the $U(N)$. Let us recall that every irreducible representation of $U(N)$ is determined, up to equivalence, by its highest weights. More explicitly, they can be parametrized by \emph{signatures} of length $N$, i.e., 
\[\widehat{U(N)}\cong \mathrm{Sign}_N:=\{\lambda=(\lambda_1\geq \cdots \geq \lambda_N)\in \mathbb{Z}^N\}.\]
For any $\lambda\in \mathrm{Sign}_N$, the irreducible character $f_\lambda$  is a function on $U(N)$ given as the normalized trace of the irreducible representation associated with $\lambda$. Moreover, we have 
\[f_\lambda(u)=\frac{s_\lambda(z_1, \dots, z_N)}{s_\lambda(1, \dots, 1)} \quad (u\in U(N)),\] 
where $z_1, \dots, z_N$ are eigenvalues of $u$, and $s_\lambda(z_1, \dots, z_N)$ is the \emph{Schur polynomial} given by 
\[s_\lambda(z_1, \dots, z_N):=\frac{\det\left[z_i^{\lambda_j+N-j}\right]_{i, j=1}^N}{\det\left[z_i^{N-j}\right]_{i, j=1}^N}.\]
Since the numerator in the right-hand side is skew-symmetric in $z_1, \dots, z_N$, the Schur polynomial $s_\lambda(z_1, \dots, z_N)$ is a Laurent polynomial in general, but it is a polynomial when $\lambda_N\geq 0$.

For every extreme character $f_\omega\in \mathcal{E}(U(\infty))$, the following approximation formula is known (see \cite{BO12,Boyer92,VK82}): there exists a sequence $(\lambda(N))_{N=1}^\infty\in \prod_{N=1}^\infty\mathrm{Sign}_N$ such that for any $n\geq 1$
\begin{equation}\label{eq:app_ext_char}
    f_\omega|_{U(n)}=\lim_{N\to \infty; N\geq n}f_{\lambda(N)}|_{U(n)}
\end{equation} 
holds, where the right-hand side converges uniformly on $U(n)$. Moreover, the parameter $\omega$ is given through the modified Frobenius coordinates of $\lambda{(N)}$ as follows: we suppose that
\[\lambda(N)=(\lambda^{+}_1(N), \lambda^{+}_2(N), \dots, -\lambda^{-}_2(N), -\lambda^{-}_1(N)),\]
where $\lambda^+_1(N), \lambda^+_2(N), \dots\geq 0$ and $\lambda^-_1(N), \lambda^-_2(N)\cdots>0$, i.e., $\lambda^+(N)=(\lambda^+_1(N), \lambda^+_2(N), \dots)$ and $\lambda^-(N)=(\lambda^-_1(N), \lambda^-_2(N), \dots)$ are positive and negative parts of $\lambda(N)$. Their \emph{modified Frobenius coordinates} are given by 
\[a^\pm_i(N):=\lambda^\pm_i(N)-i+\frac{1}{2}, \quad b^\pm_i(N):=\lambda^{\pm'}_i(N)-i+\frac{1}{2},\]
where $\lambda^{\pm '}(N)$ is the transposed Young diagram of $\lambda^\pm(N)$. Let $|\lambda^\pm(N)|:=\lambda^\pm_1(N)+\lambda^\pm_2(N)+\cdots$. Then, $\omega=(\alpha^+, \alpha^-, \beta^+, \beta^-, \gamma^+, \gamma^-)\in \Omega$ is given by 
\begin{equation}\label{eq:app_para}
    \lim_{N\to \infty}\frac{a^\pm_i(N)}{N}=\alpha^\pm_i, \quad \lim_{N\to \infty} \frac{b^\pm_i(N)}{N}=\beta^\pm_i, \quad \lim_{N\to\infty}\frac{|\lambda^\pm(N)|}{N}=\delta^\pm,
\end{equation}
\[\gamma^\pm=\delta^\pm-\sum_{i=1}^\infty(\alpha^\pm_i+\beta^\pm_i).\]
By \cite[Theorem 1.2]{OO98}, two convergences in Equations \eqref{eq:app_ext_char}, \eqref{eq:app_para} are equivalent. 

Let $N_1<N_2<\cdots$ be an increasing sequence tending to infinity. Following \cite{OO98}, we call $(\lambda(N_L))_{L=1}^\infty \in \prod_{L=1}^\infty \mathrm{Sign}_{N_L}$ a \emph{Vershik--Kerov sequence} (converging to $\omega$) if Equation \eqref{eq:app_para} holds.

\medskip
Let $\mathfrak{gl}_N$ be the complexification of the Lie algebra of $U(N)$ and $U(\mathfrak{gl}_N)$ its universal enveloping algebra, i.e., $U(\mathfrak{gl}_N)$ is a universal complex algebra generated by unit 1 and $E_{i, j}$ ($i, j=1, \dots, N$) satisfying that for all $i, j, k, l=1, \dots, N$,
\[E_{i, j}E_{k, l}-E_{k, l}E_{i, j}=\delta_{j, k}E_{i, l}-\delta_{l, i}E_{k, j}.\] 
We remark that every finite-dimensional representation of $U(N)$ extends to a representation of $U(\mathfrak{gl}_N)$, and the irreducibility is inherited. Namely, for any $\lambda \in \mathrm{Sign}_N$, the associated irreducible representation $(\pi_\lambda, V_\lambda)$ of $U(N)$ gives rise to an irreducible representation $(\widetilde\pi_\lambda, V_\lambda)$ of $U(\mathfrak{gl}_N)$.

\medskip
If $Z$ belongs to the center $Z(\mathfrak{gl}_N)$ of $U(\mathfrak{gl}_N)$, then for all $\lambda\in \mathrm{Sign}_N$ there exists a constant $f_Z(\lambda)\in \mathbb{C}$ such that $\widetilde\pi_\lambda(Z)=f_Z(\lambda)1_{V_\lambda}$, and it is known that $f_Z(\lambda)$ can be expressed as a shifted symmetric polynomial in $\lambda_1, \dots, \lambda_N$, i.e., $f_Z(\lambda)$ is symmetric in $\lambda_1+N-1, \lambda_2+N-2, \dots, \lambda_N$.

Let $\mathrm{Sign}_N^+:=\{\mu\in \mathrm{Sign}_N\mid \mu_N\geq 0\}$. For any $\mu\in \mathrm{Sign}_N^+$, the \emph{shifted Schur polynomial} $s^*_\mu(x_1,\dots, x_N)$ is defined in \cite{OO97} by 
\[s^*_\mu(x_1, \dots, x_N):=\frac{\det\left[(x_i+N-i)^{\downarrow \mu_j+N-j}\right]_{i, j=1}^N}{\det\left[(x_i+N-i)^{\downarrow N-j}\right]_{i, j=1}^N},\]
where 
\[x^{\downarrow k}:=\begin{dcases}
    x(x-1)\cdots (x-k+1) & k\geq 1, \\ 1 & k=0.
\end{dcases}\]
They play an important role in the representation theory of $U(N)$ and $U(\mathfrak{gl}_N)$. In fact, there exists a basis $\{\mathbb{S}_{\mu|N}\}_{\mu\in \mathrm{Sign}_N^+}$ of $Z(\mathfrak{gl}_N)$ such that $f_{\mathbb{S}_{\mu|N}}=s^*_\mu$, i.e.,
\[\widetilde\pi_\lambda(\mathbb{S}_{\mu|N})=s^*_\mu(\lambda)1_{V_\lambda} \quad (\lambda\in \mathrm{Sign}_N).\]
These elements $\{\mathbb{S}_{\mu|N}\}_{\mu\in \mathrm{Sign}_N^+}$ are called \emph{quantum immanants} (see \cite[Section 2]{OO97}).

\medskip
By \cite[Theorem 1.2]{OO98}, a sequence $(\lambda(N))_{L=1}^\infty$ of signatures is a Vershik--Kerov sequence if and only if for any $n\geq 1$ and $\mu\in \mathrm{Sign}_n^+$, the limit 
\[\lim_{N\to\infty}\frac{s^*_\mu(\lambda(N))}{N^{|\mu|}}\] 
exists. Moreover, by \cite[Theorem 3.1]{OO98}, this limit is equal to $s_\mu(\omega)$, which is given by the Jacobi--Trudi formula $s_\mu(\omega)=\det[h_{\mu_i-i+j}(\omega)]_{i, j=1}^n$ and the generating function 
\begin{equation}\label{eq:specialization}
    \sum_{k=0}^\infty h_k(\omega)t^k=\Phi_\omega(1+t).
\end{equation}

\subsection{Conservation operator processes derived from unitary groups}\label{sec:CLT_unitary_group}
In this section, we discuss conservation operator processes derived by irreducible representations of $U(\mathfrak{gl}_N)$. Throughout this section, we fix an increasing sequence $N_1<N_2<\cdots$ and a Vershik--Kerov sequence $(\lambda(N_L))_{L=1}^\infty \in \prod_{L=1}^\infty \mathrm{Sign}_{N_L}$. Moreover, $(\pi_
{N_L}, V_{N_L})$ denotes the irreducible representation of $U(N_L)$ associated with $\lambda(N_L)$. Let $\psi_{N_L}\in V_{N_L}$ be a unit vector. Similar to Section \ref{sec:general}, we set $(\psi_{N_L})_t:=\mathbbm{1}_{[0, t)}\otimes \psi_{N_L} \in \mathfrak{h}_{N_L}:=L^2(\mathbb{R}_{\geq 0}; V_{N_L})$, and $\langle \, \cdot\,\rangle_{N_L, t}$ and $\mathrm{Var}_{N_L, t}$ are defined in the same way.

Let us recall that $(\pi_{N_L}, V_{N_L})$ extends to an irreducible representation $(\widetilde\pi_{N_L}, V_{N_L})$ of $U(\mathfrak{gl}_{N_L})$. Thus, every $X\in U(\mathfrak{gl}_{N_L})$ gives a conservation operator process $(\Lambda_t(X))_{t\geq 0}$ by 
\[\Lambda_t(X):=\Lambda_t(1\otimes \widetilde \pi_{N_L}(X)).\] 
In particular, for any $Z\in Z(\mathfrak{gl}_{N_L})$, by Lemma \ref{lem:variance} and Equations \eqref{eq:inn_cons}, \eqref{eq:inn_cons2}, we have 
\[\langle \Lambda_t(Z)\rangle_{N_L, t}=f_Z(\lambda(N_L))t, \quad \mathrm{Var}_{\psi_{N_L, t}}(\Lambda_t(Z))=|f_Z(\lambda(N_L))|^2t.\]

\begin{remark}
    Let $Z_1, \dots, Z_m\in Z(\mathfrak{gl}_{N_L})$ be self-adjoint, that is, $f_{Z_1}(\lambda(N_L)), \dots, f_{Z_m}(\lambda(N_L))$ are real numbers. By Remark \ref{rem:indep_inc}, under the state $\langle \,\cdot\, \rangle_{N_L, t}$, the conservation operator processes $(\Lambda_t(Z_1))_{t\geq 0}, \dots, (\Lambda_t(Z_m))_{t\geq 0}$ have independent increments. Moreover, for any $0<t_1<\cdots <t_m$ and $u_1, \dots, u_m\in \mathbb{R}$, we have 
    \begin{align*}
        \left\langle e^{\mathrm{i}\sum_{j=1}^m u_j \Lambda_{t_j}(Z_j)}\right\rangle_{N_L, t_m}
        &=\prod_{k=1}^m \left\langle e^{\mathrm{i}\sum_{j=k}^m u_j \Lambda_{t_k-t_{k-1}}(Z_j)} \right\rangle_{N_L, t_k-t_{k-1}}\\
        &=\prod_{k=1}^m \left\langle \lambda (e^{\mathrm{i}\sum_{j=k}^m u_j(1\otimes Z_j)_{t_k-t_{k-1}}})\right\rangle_{N_L, t_k-t_{k-1}}\\
        &=\prod_{k=1}^m \exp\left((t_k-t_{k-1})\left(e^{\mathrm{i}\sum_{j=k}^m u_j f_{Z_j}(\lambda(N_L))}-1\right)\right),
    \end{align*}
    where $t_0:=0$ and $\lambda(\,\cdot\,)$ denotes the second quantization of a unitary operator on $\mathfrak{h}_{N_L}$. It implies that $(\Lambda_t(Z_1))_{t\geq 0}, \dots, (\Lambda_t(Z_m))_{t\geq 0}$ has the same time-ordered moments as the multivariate compound Poisson process with jumps of size $f_{Z_1}(\lambda(N_L)), \dots, f_{Z_m}(\lambda(N_L))$.
\end{remark}

We assume that the Vershik--Kerov sequence $(\lambda(N_L))_{L=1}^\infty$ converges to $\omega\in \Omega$ (see Equation \eqref{eq:app_para}) and $L/N_L\to 1$ as $L\to \infty$. As we discussed in the previous section, the following law of large numbers holds true:

\begin{proposition}\label{prop:LLN}
    For any integer partition $\mu$, we have 
    \[\lim_{L\to\infty}\frac{\langle \Lambda_t(\mathbb{S}_{\mu|N_L})\rangle_{N_L, t}}{N_L^{|\mu|}}=s_\mu(\omega)t,\]
    where the right-hand side is given by Equation \eqref{eq:specialization}.
\end{proposition}

By the same argument of Theorem \ref{thm:general_CLT}, we further obtain the following central limit theorem:
\begin{theorem}\label{thm:CLT_unitary}
    Let $\mu_1, \dots, \mu_m$ be integer partitions and
    \[X^{(i)}_{N_L, t}:=\Lambda_t(\mathbb{S}_{\mu_i|N_L}), \quad \widetilde X^{(i)}_{N_L, t}:=\frac{X^{(i)}_{N_L, t}-\langle X^{(i)}_{N_L, t}\rangle_{N_L, t}}{N_L^{|\mu_i|+1/2}} \quad (i=1, \dots, m).\]
    Then, the family of processes $(\widetilde X^{(1)}_{N_L, Lt})_{t\geq 0}, \dots , (\widetilde X^{(m)}_{N_L, Lt})_{t\geq 0}$ with scaled time parameter $Lt$ converges in the moment sense to a Gaussian family as $L\to \infty$, i.e.,
    \[\lim_{L\to \infty}\langle \widetilde X^{(1)}_{N_L, Lt}\cdots \widetilde X^{(m)}_{N_L, Lt}\rangle_{N_L, Lt}=\begin{dcases}
        \sum_{\pi \in \mathcal{P}_2(m)}\prod_{\{i, j\}\in \pi} ts_{\mu_i}(\omega)s_{\mu_j}(\omega) & m \text{ is even}, \\ 0 & m \text{ is odd}.
    \end{dcases}\]
    The covariance of the Gaussian family in the large $L$ limit is given by $ts_{\mu_i}(\omega)s_{\mu_j}(\omega)$ for every $i, j=1, \dots, m$.
\end{theorem}
\begin{proof}
    By the same discussion in the proof of Theorem \ref{thm:general_CLT}, we have 
    \begin{align*}
        &\langle \widetilde X^{(1)}_{N_L, Lt}\cdots \widetilde X^{(m)}_{N_L, Lt}\rangle_{N_L, Lt}\\
        &=\frac{1}{N_L^{m/2}}\sum_\pi L(L-1)\cdots (L-|\pi|+1)\prod_{\{i_1<\cdots<i_k\}\in \pi}\langle \underline{X}^{(i_1)}_{N_L, t}\cdots \underline{X}^{(i_k)}_{N_L, t}\rangle_{N_L, t},
    \end{align*}
    where $\underline{X}^{(i)}_{N_L, t}:=N_L^{1/2}\widetilde{X}^{(i)}_{N_L, t}$, and the summation is over all partitions of $\{1, \dots, m\}$ into at most $L$ blocks. Since the underlying operator $\widetilde \pi_N(\mathbb{S}_{\mu_i|N_L})$ of $X^{(i)}_{N_L, t}$ is a scalar operator on $V_{N_L}$, each factor $\langle \underline{X}^{(i_1)}_{N, t}\cdots \underline{X}^{(i_k)}_{N_L, t}\rangle_{N_L, t}$ is a polynomial of $s_{\mu_i}(\lambda(N_L))/N_L^{|\mu_i|}$ ($i=1, \dots, m$). Thus, it converges to a polynomial of $s_{\mu_i}(\omega)$. Since $L/N_L\to 1$ as $L\to\infty$, if $|\pi|<m/2$, the associated term converges to zero as $L\to \infty$. On the other hand, $\pi$ should have a singleton if $|\pi|>m/2$, and hence, the associated term vanishes. Thus, the remaining term has to satisfy $|\pi|=m/2$, i.e., $m$ is even, and $\pi$ is a pair partition. Finally, since $\lim_{L\to \infty}\langle \underline{X}^{(i)}_{N_L, t} \underline{X}^{(j)}_{N_L, t} \rangle_{N_L, t}=ts_{\mu_i}(\omega)s_{\mu_j}(\omega)$, we obtain the desired formula.
\end{proof}

%%%%%%%%%%%%%%%%%%%%%%%%%%%%%%%%%%%%%%%%%%%%%%%%%%%%%%%%%%%%%%%%%%%%%%%%%%%%%%%%%%%%%%%%%%%%%%%%%%%%%%%%%%%%%%%%%%%%%%%%%%%%%%%%%%%%%%%%%%%%%%%%%%
\subsection{Comments on the symmetric group case}
We can apply the same argument as in the previous section to the symmetric groups rather than unitary groups. Here, we mainly refer to the textbook \cite{BO:book} on the asymptotic representation theory of the symmetric groups.

Let $S(N)$ be the symmetric group of degree $N$. The \emph{infinite symmetric group} $S(\infty)$ is defined by $\varinjlim_N S(N)$ and naturally identified with the group of finite permutations on $\{1, 2, \dots\}$. The complete classification of extreme characters of $S(\infty)$ is well known as Thoma's theorem (see \cite[Corollary 4.2]{BO:book}), and they are parametrized by the Thoma simplex $\Delta$. Here, $\Delta$ is the set of $\omega=(\alpha, \beta)\in [0, 1]^\infty\times [0, 1]^\infty$ satisfying 
\[\alpha=(\alpha_1\geq \alpha_2\geq \cdots), \quad \beta=(\beta_1\geq \beta_2\geq \cdots),\]
\[\sum_{i=1}^\infty(\alpha_i+\beta^i)\leq 1.\]
Moreover, for any $\omega\in \Delta$, the associated extreme character, denoted by $\chi^\omega$, is given by 
\[\chi^\omega(\sigma)=\prod_{k=2}^\infty\left(\sum_{i=1}^\infty \alpha_i^k+(-1)^{k-1}\beta_i^k\right)^{m_k} \quad (\sigma\in S(\infty)),\]
where $m_k$ is the number of cycle permutations of length $k$ in the cycle decomposition of $\sigma$. We remark that cycle decomposition provides a correspondence between the conjugacy classes of $S(N)$ and the Young diagrams with $N$ boxes. Similarly, the conjugacy classes of $S(\infty)$ correspond to the set $\mathbb{Y}^\circ$ of Young diagrams $\rho=(\rho_1, \rho_2, \dots)$ such that $\rho_i\neq 1$ for all $i\geq 1$. For any $\rho\in \mathbb{Y}^\circ$, we denote by $\chi^\omega_\rho$ the value of $\chi^\omega$ on the associated conjugacy class.

Similar to Equation \eqref{eq:app_ext_char}, the approximation formula of extreme characters by irreducible characters of $S(N)$ is known as follows (see \cite[Theorem 6.16]{BO:book}): let us recall that all irreducible representations of $S(N)$ can be parametrized by the set of Young diagrams with $N$ boxes, i.e., 
\[\widehat{S(N)}\cong \mathbb{Y}_N:=\{\lambda=(\lambda_1\geq \lambda_2\geq\cdots)\in \mathbb{Z}^\infty_{\geq 0}\mid |\lambda|:=\lambda_1+\lambda_2+\cdots =N\}.\]
For every $\lambda\in \mathbb{Y}_N$ we denote by $\chi^\lambda$ the associated irreducible character of $S(N)$. Here, we normalize it by $\chi^\lambda(e)=1$. Then, for every extreme character $\chi$ of $S(\infty)$, there exists a sequence $(\lambda(N))_{N= 1}^\infty\in \prod_{N=1}^\infty \mathbb{Y}_N$ such that for all $n\geq 1$
\begin{equation}\label{eq:app_ext_char_S}
    \chi|_{S(n)}=\lim_{N\to\infty; N\geq n}\chi^{\lambda(N)}|_{S(n)}.
\end{equation}
Moreover, similarly to the unitary group case, the corresponding parameter $\omega=(\alpha, \beta)\in \Delta$ is given by the modified Frobenius coordinates of $\lambda(N)$. See Equation \eqref{eq:app_para}. In this case, we say that $(\lambda(N))_{N\geq 1}$ converges to $\omega$. Moreover, it is equivalent to Equation \eqref{eq:app_ext_char_S} with $\chi=\chi^\omega$.

\medskip
For any $\rho\in \mathbb{Y}^\circ$ with $|\rho|\leq N$ we obtain the Young diagram $\rho\cup (1^{N-|\rho|})$ with $N$ boxes by adding $N-|\rho|$ rows to $\rho$. We denote by $C_{\rho\cup \{1^{N-|\rho|}\}}$ the associated conjugacy class of $S(N)$ and define $A_{\rho|N} \in \mathbb{C}[S(N)]$ by
\[A_{\rho|N }:=\frac{1}{|C_{\rho\cup\{1^{N-|\rho|}\} }|}\sum_{g\in C_{\rho\cup\{1^{N-|\rho|}\}}}g.\]
By definition, $A_{\rho|N}$ belongs to the center $Z(\mathbb{C}[S(N)])$ of $\mathbb{C}[S(N)]$. Moreover, $\{A_{\rho|N}\}_{\rho \in \mathbb{Y}^\circ; |\rho|\leq N}$ form a basis of $Z(\mathbb{C}[S(N)])$. 

We now consider conservation operator processes derived from them. Let $N_1<N_2<\cdots$ be an increasing sequence tending to infinity and assume that a sequence $(\lambda(N_L))_{L\geq 1}$ of Young diagrams converges to $\omega\in \Delta$ in the sense of Equation \eqref{eq:app_para}. For every $L\geq 1$ we denote by $(\sigma_{N_L}, W_{N_L})$ the irreducible representation of $S(N_L)$ corresponding to $\lambda(N_L)$ and fix a unit vector $\psi_{N_L}\in W_{N_L}$. Moreover, for all $t\geq 0$ we define $(\psi_{N_L})_t\in \mathfrak{h}_{N_L}:=L^2(\mathbb{R}_{\geq 0}; W_{N_L})$ and $\langle\, \cdot\, \rangle_{N_L, t}$ by the same way in the previous section.

Since $(\sigma_{N_L}, W_{N_L})$ naturally extends to a representation of $\mathbb{C}[S(N_L)]$, every $A\in \mathbb{C}[S(N_L)]$ gives a conservation operator process $(\Lambda_t(A))_{t\geq 0}$ on $\mathcal{F}(\mathfrak{h}_{N_L})$ by $\Lambda_t(A):=\Lambda_t(1\otimes \sigma_{N_L}(A))$.

By definition, for every $A\in \mathbb{C}[S(N_L)]$, we have $\langle \Lambda_t(A)\rangle_{N_L, t}=t\chi^{\lambda(N_L)}(A)$. In particular, we have $\langle \Lambda_t(A_{\rho|N_L})\rangle_{N_L, t}=t\chi^{\lambda(N_L)}_{\rho\cup\{1^{N_L-|\rho|}\}}$ for any $\rho\in \mathbb{Y}^\circ$ with $|\rho|\leq N_L$.

Equation \eqref{eq:app_ext_char_S} implies the following law of large numbers:
\begin{proposition}
    For any $\rho\in \mathbb{Y}^\circ$ we have $\lim_{L\to \infty} \langle \Lambda_t(A_{\rho|N_L})\rangle_{N_L, t}=t\chi^\omega_\rho$.
\end{proposition}

Moreover, as in Theorem \ref{thm:CLT_unitary}, the following central limit theorem holds true.
\begin{theorem}
    Let $\rho_1, \dots, \rho_m\in \mathbb{Y}^\circ$ and 
    \[X^{(i)}_{N_L, t}:=\Lambda_t(A_{\rho_i|N_L}), \quad \widetilde{X}^{(i)}_{N_L, t}:=\frac{X^{(i)}_{N_L, t}-\langle X^{(i)}_{N_L, t}\rangle_{N_L, t}}{N_L^{1/2}}.\]
    Then, the family of processes $(\widetilde X^{(1)}_{N_L, Lt})_{t\geq 0}, \dots , (\widetilde X^{(m)}_{N_L, Lt})_{t\geq 0}$ with scaled time $Lt$ converges in the moment sense to a Gaussian family as $L\to \infty$, i.e.,
    \[\lim_{L\to \infty}\langle \widetilde X^{(1)}_{N_L, Lt}\cdots \widetilde X^{(m)}_{N_L, Lt}\rangle_{N_L, Lt}=\begin{dcases}
        \sum_{\pi \in \mathcal{P}_2(m)}\prod_{\{i, j\}\in \pi}t\chi^\omega_{\rho_i} \chi^\omega_{\rho_j} & m \text{ is even}, \\ 0 & m \text{ is odd}.
    \end{dcases}\]
    The covariance of the Gaussian family in the large $L$ limit is given by $t\chi^\omega_{\rho_i} \chi^\omega_{\rho_j}$ for every $i, j=1, \dots, m$.
\end{theorem}

As we mentioned, the approximation of extreme characters of $S(\infty)$ (see Equation \eqref{eq:app_ext_char_S}) is equivalent to the convergence of the associated modified Frobenius coordinates. In the literature, CLTs for extreme characters of $S(\infty)$ have been provided in \cite{Meliot} (also for the infinite Hecke algebra) and in \cite{Bufetov12}. There are also works developing quantum probabilistic approach (quantum decompositions on Fock spaces) for the CLT-type results (and these Jack deformations) in \cite{HO:book}. Compared with these previous results, our result establishes a CLT for time-parametrized linear statistics of central elements $A_{\rho|N}$ in $\mathbb{C}[S(N)]$ within the framework of conservation operator processes. Here, independent and stationary increments play an essential role. The limiting Gaussian covariance is given by $t\chi^\omega_{\rho_i}\chi^\omega_{\rho_j}$.

%%%%%%%%%%%%%%%%%%%%%%%%%%%%%%%%%%%%%%%%%%%%%%%%%%%%%%%%%%%%%%%%%%%%%%%%%%%%%%%%%%%%%%%%%%%%%%%%%%%%%%%%%%%%%%%%%%%%%%%%%%%%%%%%%%%%%%%%%%%%%%%%%%
\section{Conservation operator processes derived from quantum unitary groups}\label{sec:quantum_unitary_groups}
\subsection{Asymptotic representation theory of quantum unitary groups}

In the previous section, we studied conservation operator processes from unitary groups. Now, we turn to the case of \emph{quantum} unitary groups. Similar to the previous section, we need the results of the asymptotic representation theory for quantum unitary groups. See \cite{Gorin12,Sato1,Sato3}.

\medskip
Throughout the paper, we assume that a quantization parameter $q$ is in $(0, 1)$. Let $U_q(\mathfrak{gl}_N)$ denote the quantum universal enveloping algebra associated with $\mathfrak{gl}_N$. See \cite[Section 6.1]{KS:book} for the definition. It is well known that $U_q(\mathfrak{gl}_N)$ has the same representation theory as $U(\mathfrak{gl}_N)$. More precisely, any type-1 irreducible representations of $U_q(\mathfrak{gl}_N)$ precisely correspond to $\mathrm{Sign}_N$. Furthermore, for every $\lambda\in \mathrm{Sign}_N$, the associated irreducible representation, denoted by $(T_\lambda, V_\lambda)$, has the same dimension as the irreducible representation $(\tilde \pi_\lambda, V_\lambda)$ of $U(\mathfrak{gl}_N)$.

Let $Z_q(\mathfrak{gl}_N)$ denote the center of $U_q(\mathfrak{gl}_N)$. Since $(T_\lambda, V_\lambda)$ is irreducible, for any $Z\in Z_q(\mathfrak{gl}_N)$ its representation $T_\lambda(Z)$ is a scalar operator. Moreover, such a scalar is described by a symmetric polynomial in $q^{2\lambda_1}, q^{2(\lambda_2-1)}, \dots, q^{2(\lambda_N-N+1)}$, and it coincides with the Harish-Chandra image of $Z$. Similar to the previous section, we are interested in the asymptotic behavior of those constants as $N\to \infty$.

Before that, we introduce the factorial Schur polynomials. Let $a=(a_j)_{j=1}^\infty$ be a sequence of parameters. For any $\mu\in \mathrm{Sign}_N^+$, the \emph{factorial Schur polynomial} $s_\mu(x_1, \dots, x_N | a)$ is defined by 
\[s_\mu(x_1, \dots, x_N | a):=\frac{\det\left[(x_i | a)^{\mu_j+N-j}\right]_{i, j=1}^N}{\det\left[(x_i | a)^{N-j}\right]_{i, j=1}^N},\]
where 
\[(x| a)^k:=\begin{dcases}(x+a_1)\cdots (x+a_k) & k\geq 1, \\ 1 & k=0.\end{dcases}\]
For instance, we have $s_\mu(x_1, \dots, x_N | a)= s^*_\mu(x_1, \dots, x_N)$ if $a=(-j+1)_{j=1}^\infty$. Following \cite{Gorin12}, for any $\mu\in \mathrm{Sign}_N^+$ the \emph{$q$-interpolation Schur polynomial} $s^*_\mu(x_1, \dots, x_N; q)$ is defined by
\[s^*_\mu(x_1, \dots, x_N; q):=s_\mu(x_1, \dots, x_N| (-q^{j-N})_{j=1}^\infty).\]

For fixed complex parameter $a$, the factorial Schur polynomials also form a basis of the $\mathbb{C}$-algebra of symmetric polynomials in $x_1, \dots, x_N$ . Therefore, there exists a basis $\{\widehat{\mathbb{S}}^{(q)}_{\mu|N}\}_{\mu\in \mathrm{Sign}_N^+}$ of $Z_q(\mathfrak{gl}_N)$ such that for every $\lambda\in \mathrm{Sign}_N$ we have
\[T_\lambda(\widehat{\mathbb{S}}^{(q)}_{\mu|N})=s^*_\mu(q^{2\lambda_1}, q^{2(\lambda_2-1)} \dots, q^{2(\lambda_N-N+1)}; q^2).\]
Notably, another basis of $Z_q(\mathfrak{gl}_N)$ was discussed in \cite{JLM}, and their irreducible representations are described by the factorial Schur polynomials with parameter $a=(zq^{-2(j-1)})_{j=1}^\infty$ and $z\in \mathbb{C}$.

\medskip
We now investigate the asymptotic behavior of irreducible representations of $Z_q(\mathfrak{gl}_N)$. Let $N_1<N_2<\cdots$ be an increasing sequence tending to infinity. Following \cite{Gorin12}, we say that $(\lambda(N_L))_{L=1}^\infty\in \prod_{L=1}^\infty \mathrm{Sign}_{N_L}$ \emph{stabilize} to $\nu=(\nu_1\leq \nu_2\leq\cdots)\in \mathbb{Z}^\infty$ if for every $j\geq 1$
\[\lim_{L\to \infty} \lambda(N_L)_{N_L+1-j}=\nu_j.\]

Let $\delta_n:=(0, 1, \dots, n-1)$. Moreover, we define $q^\alpha:=(q^{\alpha_1}, \dots, q^{\alpha_n})$ for any $\alpha=(\alpha_1, \cdots, \alpha_n)$. 

\medskip
The following is essentially proved in \cite{Gorin12}:
\begin{proposition}\label{prop:asymptotics_q_immanants}
    If $(\lambda(N_L))_{L=1}^\infty\in \prod_{L=1}^\infty \mathrm{Sign}_{N_L}$ stabilize to $\nu=(\nu_1\leq \nu_2\leq\cdots)\in \mathbb{Z}^\infty$, then for any integer partition $\mu$ with length $n$, 
    \[\lim_{L\to \infty}\frac{s^*_\mu(q^{2(\lambda(N_L)-\delta_{N_L})}; q^2)}{q^{-2(N_L-1)|\mu|}}=\frac{s^*_\mu(q^{2\mu}; q^2)}{s^*_\mu(q^{2(\mu-\delta_n)}; q^2)}s_\mu(\omega_\nu),\]
    where $s_\mu(\omega_\nu)$ is determined by the Jacobi--Trudi formula $s_\mu(\omega_\nu)=\det[h_{\mu_i-i+j}(\omega_\nu)]_{i, j=1}^n$ and the generating function  
    \[\sum_{k=0}^\infty h_k(\omega_\nu)t^k = \frac{\prod_{j=0}^\infty(1-q^{2j}t)}{\prod_{j=1}^\infty(1-q^{2(\nu_j+j-1)t})}.\]
\end{proposition}
\begin{proof}
    By \cite[Theorem 1.3(1), Proposition 5.9]{Gorin12}, for any $n\geq 1$ we have
    \begin{align*}
        &\lim_{L\to \infty; N_L\geq n}\frac{s_{\lambda(N_L)}(x_1, \dots, x_n, q^{-2n}, \dots, q^{-2(N_L-1)})}{s_{\lambda(N_L)}(q^{-2\delta_N})}\\
        &=\sum_{\mu\in \mathrm{Sign}_{N_L}^+}(-1)^{|\mu|}q^{2(n(\mu)-n(\mu'))}s_\mu(\omega_\nu)s^*_\mu(x_1, \dots, x_n; q^2),
    \end{align*}
    where $\mu'$ is the transposed Young diagram of $\mu$, and $n(\mu):=\sum_{j=1}^n(j-1)\mu_j$. Here, the left-hand side converges uniformly on $\{(x_1, \dots, x_n)\in \mathbb{C}^n\mid |x_j|=q^{2(1-j)}\, (j=1, \dots, n)\}$. By the binomial formula (see \cite[Equation (17)]{Gorin12}), we have
    \begin{align*}
        &\frac{s_{\lambda(N_L)}(x_1, \dots, x_n, q^{-2n}, \dots, q^{-2(N_L-1)})}{s_{\lambda(N_L)}(q^{-2\delta_{N_L}})}\\
        &=\sum_{\mu\in \mathrm{Sign}_{N_L}^+}\frac{s^*_\mu(q^{2(\lambda(N_L)-\delta_{N_L})}; q^2)}{q^{-2(N_L-1)|\mu|}s^*_\mu(q^{2(\mu-\delta_n)}; q^2)}\frac{s^*_\mu(x_1, \dots, x_n; q^2)}{s_\mu(q^{2\delta_{N_L}})}.
    \end{align*}
    Thus, \cite[Proposition 6.3]{Gorin12}, the convergence of each coefficient occurs, that is,
    \[\lim_{L\to \infty; N_L\geq n}\frac{s^*_\mu(q^{2(\lambda(N_L)-\delta_{N_L})}; q^2)}{q^{-2(N_L-1)|\mu|}}\frac{1}{s^*_\mu(q^{2(\mu-\delta_n)}; q^2)s_\mu(q^{2\delta_{N_L}})}=(-1)^{|\mu|}q^{n(\mu)-n(\mu')}s_\mu(\omega_\nu).\]
    The assertion follows from $\lim_{L\to \infty}s^*_\mu(q^{2\mu}; q^2)s_\mu(q^{2\delta_{N_L}})=(-1)^{|\mu|}q^{2(n(\mu')-n(\mu))}$ (see \cite[Proof of Proposition 5.9]{Gorin12}).
\end{proof}

%%%%%%%%%%%%%%%%%%%%%%%%%%%%%%%%%%%%%%%%%%%%%%%%%%%%%%%%%%%%%%%%%%%%%%%%%%%%%%%%%%%%%%%%%%%%%%%%%%%%%%%%%%%%%%%%%%%%%%%%%%%%%%%%%%%%%%%%%%%%%%%%%%
\subsection{Conservation operator processes derived from quantum unitary groups}

Here, we demonstrate that our result in Theorem \ref{thm:general_CLT} is also applicable to the quantum group case.

Throughout this section, we fix an increasing sequence $N_1<N_2<\cdots$ tending to infinity and a sequence $(\lambda(N_L))_{L=1}^\infty\in \prod_{L=1}^\infty \mathrm{Sign}_{N_L}$ that stabilize to $\nu=(\nu_1\leq \nu_2\leq \cdots)\in \mathbb{Z}^\infty$. For every $L\geq 1$, we denote by $(T_{N_L}, V_{N_L})$ the irreducible representation of $U_q(\mathfrak{gl}_{N_L})$ associated with $\lambda(N_L)$. Similar to Section \ref{sec:CLT_unitary_group}, let $\psi_{N_L}\in V_{N_L}$ be a unit vector, and $\langle \,\cdot\,\rangle_{N_L, t}$ is defined in the same way in Section \ref{sec:CLT_unitary_group}.

For any integer partition $\mu$ with length $n$, we define $Z^{(q)}_{\mu|N_L}\in Z_q(\mathfrak{gl}_{N_L})$ by 
\[Z^{(q)}_{\mu|N_L}:=\frac{s^*_\mu(q^{2(\mu-\delta_n)}; q^2)}{s^*_\mu(q^{2\mu}; q^2)}\widehat{\mathbb{S}}^{(q)}_{\mu|N_L}.\]
By Proposition \ref{prop:asymptotics_q_immanants}, we obtain the following law of large numbers:
\begin{proposition}
    For any integer partition $\mu$, we have 
    \[\lim_{L\to \infty}\frac{\langle \Lambda_t(Z^{(q)}_{\mu|N_L})\rangle_{N_L, t}}{q^{-2(N_L-1)|\mu|}}=s_\mu(\omega_\nu)t.\]
\end{proposition}

Furthermore, as in Theorem \ref{thm:CLT_quantum_unitary}, the following central limit theorem holds true:
\begin{theorem}\label{thm:CLT_quantum_unitary}
    Let $\mu_1, \dots, \mu_m$ be integer partitions and 
    \[X^{(i)}_{N_L, t}:=\Lambda_t(Z^{(q)}_{\mu_i|N_L}), \quad \widetilde X^{(i)}_{N_L, t}:=\frac{X^{(i)}_{N_L, t}-\langle X^{(i)}_{N_L, t}\rangle_{N_L, t}}{N_L^{1/2}q^{-2(N_L-1)|\mu|}} \quad (i=1, \dots, m).\]
    Then, the family of processes $(\widetilde X^{(1)}_{N_L, Lt})_{t\geq 0}, \dots , (\widetilde X^{(m)}_{N_L, Lt})_{t\geq 0}$ with scaled time parameter $Lt$ converges in the moment sense to a Gaussian family as $L\to \infty$, i.e.,
    \[\lim_{L\to \infty}\langle \widetilde X^{(1)}_{N_L, Lt}\cdots \widetilde X^{(m)}_{N_L, Lt}\rangle_{N_L, Lt}=\begin{dcases}
        \sum_{\pi \in \mathcal{P}_2(m)}\prod_{\{i, j\}\in \pi}ts_{\mu_i}(\omega_\nu)s_{\mu_j}(\omega_\nu) & m \text{ is even}, \\ 0 & m \text{ is odd}.
    \end{dcases}\]
    The covariance of the Gaussian family in the large $L$ limit is given by $ts_{\mu_i}(\omega_\nu)s_{\mu_j}(\omega_\nu)$ for every $i, j=1, \dots, m$.
\end{theorem}

%%%%%%%%%%%%%%%%%%%%%%%%%%%%%%%%%%%%%%%%%%%%%%%%%%%%%%%%%%%%%%%%%%%%%%%%%%%%%%%%%%%%%%%%%%%%%%%%%%%%%%%%%%%%%%%%%%%%%%%%%%%%%%%%%%%%%%%%%%%%%%%%%%
\section{Concluding remarks}

For any $t>0$, the \emph{(one-sided) Plancherel character} $f_t$ of $U(\infty)$ is defined by 
\[f_t(u)=\prod_{z}e^{t(z-1)}=e^{t\mathrm{Tr}(u-1)},\] 
where $z$ runs over all eigenvalues of $u\in U(\infty)$. By the Edrei--Voiculescu theorem, it is an extreme character of $U(\infty)$. Moreover, $f_t$ can be realized by operator-valued processes defined on a symmetric Fock space. In fact, let $\mathfrak{h}_N:=L^2(\mathbb{R}_{\geq 0}; \mathbb{C}^N)\cong L^2(\mathbb{R}_{\geq 0})\otimes \mathbb{C}^N$ and define two types of unitary operators on $\mathcal{F}(\mathfrak{h}_N)$ by 
\[\rho_t(u)e(\psi):=e((1\otimes u)_t\psi), \quad w(\varphi)e(\psi):=\exp\left(-\frac{\|\varphi\|^2}{2}-\langle \varphi, \psi\rangle\right)e(\psi+\varphi)\]
for any $u\in U(N)$ and $\varphi, \psi\in \mathfrak{h}_N$. Here, $(1\otimes u)_t:=p([0, t))(1\otimes u)$ and $p([0, t))$ is the orthogonal projection onto $L^2([0, t); \mathbb{C}^N)$. Moreover, $w(\varphi)$ is called the \emph{Weyl operator}. Let $e_1, \dots, e_N\in \mathbb{C}^N$ denote the standard basis and define $e_{j, t}:=\mathbbm{1}_{[0, t)}\otimes e_j\in \mathfrak{h}_N$ ($j=1, \dots, N$). Then, we have 
\[f_t(u)=\left\langle \left(\bigotimes_{j=1}^Nw(e_{j, t})\right)^* \rho_t(u)^{\otimes N} \left(\bigotimes_{j=1}^Nw(e_{j, t}) \right)\Omega^{\otimes N}, \Omega^{\otimes N}\right\rangle\]
for any $u\in U(N)$, where $\Omega=e(0)$ is the vacuum vector. 

In this paper, we have not considered the bialgebra structure of $U(\mathfrak{gl}_N)$, but its comultiplication $\Delta\colon U(\mathfrak{gl}_N)\to U(\mathfrak{gl}_N)\otimes U(\mathfrak{gl}_N)$ is defined by $\Delta(X):=X\otimes 1+1\otimes X$ for any $X\in \mathfrak{gl}_N$. Moreover, the representation of $U(\mathfrak{gl}_N)$ derived from $(\rho_t, \mathcal{F}(\mathfrak{h}_N))$ is given as a solution of the following quantum stochastic differential equation (QSDE):
\[dj_t(X)=(j_t\star d\Lambda_t)(X)=\sum j_t(X_1)d\Lambda_t(X_2), \quad j_0(X)=0\]
for any $X\in U(\mathfrak{gl}_N)$, where $\Delta(X)=\sum X_1\otimes X_2$. See \cite{Parthasarathy:book,Franz06} for more details.

Finally, we would like to mention that the operator-valued processes given by $j_t$ on $Z(\mathfrak{gl}_N)$ has been studied in \cite{HP1994}. It would be a natural and interesting future direction to analyze the asymptotic behavior of these processes using quantum stochastic calculus. It seems to provide an algebraic understanding of the CLT in \cite{BB13,BB14,Kuan16}. Furthermore, extending the work in \cite{HP1994} to more general processes defined by QSDE also seems a promising direction.

%%%%%%%%%%%%%%%%%%%%%%%%%%%%%%%%%%%%%%%%%%%%%%%%%%%%%%%%%%%%%%%%%%%%%%%%%%%%%%%%%%%%%%%%%%%%%%%%%%%%%%%%%%%%%%%%%%%%%%%%%%%%%%%%%%%%%%%%%%%%%%%%%%
\section*{Acknowledgements}

The author gratefully acknowledges beneficial comments from Professor Akihito Hora, Professor Takahiro Hasebe, and Professor Vadim Gorin on the draft of this paper.

}
%%%%%%%%%%%%%%%%%%%%%%%%%%%%%%%%%%%%%%%%%%%%%%%%%%%%%%%%%%%%%%%%%%%%%%%%%%%%%%%%%%%%%%%%%%%%%%%%%%%%%%%%%%%%%%%%%%%%%%%%%%%%%%%%%%%%%%%%%%%%%%%%%%

\end{document}